\documentclass[11pt, reqno, oneside]{amsart}

\usepackage[svgnames, dvipsnames, table]{xcolor}
\usepackage[colorlinks = true, allcolors = UCRceleste, pagebackref=true, colorlinks]{hyperref}

\makeatletter
\def\section{\@startsection{section}{1}%
  \z@{.7\linespacing\@plus\linespacing}{.5\linespacing}%
  {\normalfont\bfseries\raggedright}}

\def\@secnumfont{\bfseries}
\makeatother

\usepackage{etoolbox}

\patchcmd{\section}{\scshape\centering}{\bfseries\raggedright}{}{}
\patchcmd{\section}{\bfseries}{\bfseries}{}{}

\usepackage[utf8]{inputenc}
\usepackage{caption}
\usepackage{subcaption}
\usepackage{float}
\usepackage{mathrsfs}
\usepackage{xcolor}
\usepackage{mathtools}
\usepackage{seqsplit}
\usepackage{thmtools}

\usepackage{color}
\usepackage{amssymb}
\usepackage{tikz}
\definecolor{UCRceleste}{RGB}{0,192,243}
\definecolor{mintgreen}{RGB}{152,255,152}
\definecolor{pinksalmon}{RGB}{255,102,102}
\definecolor{hueso}{RGB}{245,245,220}
\definecolor{marfil}{RGB}{255,253,208}
\definecolor{amarillo}{RGB}{255,255,0}

\usepackage{geometry}

\usepackage{enumitem}

\usepackage{orcidlink}

\usepackage[ruled, linesnumbered, algosection]{algorithm2e}

\numberwithin{equation}{section}

\newtheorem*{theorem*}{Theorem}
\newtheorem{theorem}{Theorem}[section]

\newtheorem{proposition}[theorem]{Proposition}
\newtheorem{corollary}[theorem]{Corollary}

\theoremstyle{definition}

\declaretheoremstyle[headfont=\bfseries, bodyfont=\normalfont]{simplemark}
\declaretheorem[style=simplemark, name=Remark, sibling=theorem]{remark}
\declaretheorem[style=simplemark, name=Example, sibling=theorem]{example}

\declaretheoremstyle[headfont=\bfseries, bodyfont=\normalfont, qed=\qedsymbol]{simpleproof}

\declaretheorem[style=simpleproof, name=Proof, numbered=no]{proof}

\newcommand{\suchthat}{\;\ifnum\currentgrouptype=16 \middle\fi|\;}
\newcommand{\Z}{\mathbb{Z}}

\newcommand{\R}{\mathbb{R}}

\newcommand{\op}[1]{\operatorname{#1}}
\newcommand{\cal}[1]{\mathcal{#1}}

\newcommand{\ZZ}{\mathbb{Z}}
\newcommand{\set}[2]{\left\{{#1} \, : \, {#2}\right\}}
\def\abs#1{\left\vert{#1}\right\vert}
\newcommand{\lcm}{\operatorname{lcm}}
\newcommand{\val}{\operatorname{val}}
\newcommand{\tri}{\mathbf{1}}
\def\tmod#1{\;\mathrm{mod}\, {#1}}

\newcommand{\legendre}[2]{\ensuremath{\left( \frac{#1}{#2} \right) }}
\newcommand{\tlegendre}[2]{\ensuremath{( \frac{#1}{#2} ) }}

\def\abs#1{\left\vert{#1}\right\vert}

\def\dparenth#1{\left(\!\left({#1}\right)\!\right)} 

\DeclareMathOperator{\mymod}{mod\,}

\newcommand{\krop}{{\tlegendre{\cdot}{p}}}

\usepackage[english]{babel}

\begin{document}

\title{On the efficient computation of Fourier coefficients of eta-quotients}

\author[A. Barquero-Sanchez et al.]{Adrian Barquero-Sanchez\orcidlink{0000-0001-7847-2938}, Juan Pablo De Rasis\orcidlink{0000-0002-6225-4618}, Nicolás Sirolli\orcidlink{0000-0002-0603-4784} and Jean Carlos Villegas-Morales\orcidlink{0009-0004-8933-0627}}

\address{Escuela de Matem\'atica, Universidad de Costa Rica, San Jos\'e 11501, Costa Rica}
\address{Centro de Investigación en Matemática Pura y Aplicada, Universidad de Costa Rica, San Jos\'e 2060, Costa Rica}
\email{adrian.barquero\_s@ucr.ac.cr}

\address{Department of Mathematics (College of Arts and Sciences), The Ohio State University. Columbus, OH 43210-1174, USA}
\email{derasis.1@osu.edu}

\address{Departamento de Matemática, FCEyN, UBA and IMAS, CONICET -- Pabellón I, Ciudad Universitaria, Ciudad Autónoma de Buenos Aires (1428), Argentina}
\email{nsirolli@dm.uba.ar}

\address{Department of Mathematics, 3368 TAMU, Texas A\&M University, College Station, TX 77843-3368, USA}
\email{jean\_villegas.02@tamu.edu}

\begin{abstract}
	The Fourier coefficients of a negative weight eta-quotient, in many
	particular cases, and after Sussman in general,
	are known to be expressible by Hardy--Ramanujan--Rademacher type series.

	We show that the central terms of the coefficients of these series can be
	efficiently computed, showing that they can be expressed in terms of twisted
	Kloosterman sums, and that they satisfy multiplicativity relations;
	this extends the results from Lehmer for the partition function.

	We also give explicit bounds for the tails of these series, needed for
	effectively computing the aforementioned Fourier coefficients.
\end{abstract}


\maketitle

\section{Introduction}

The partition function $p(n)$ counts the number of ways, up to permutations, that
a positive integer $n$ can be written as a sum of positive integers.
In other words, it is defined by the formal identity

\begin{equation*}
	q^{-1/24}
	\sum_{n \geq 0} p(n) q^n =
	\frac{1}{\eta(q)},
\end{equation*}
where $\eta(q) = q^{1/24} \prod_{n \geq 1} (1-q^n)$ is Dedekind's eta function.

Obtaining $p(n)$ using this expression involves computing also $p(k)$ for $k < n$ and this is computationally heavy, making it unfeasible even for relatively small values of $n$. 
A more convenient formula for computing the numbers $p(n)$ individually was
given by Rademacher: improving results from Hardy and Ramanujan, he showed in
\cite{R38} that
\begin{equation}
	\label{eqn:partitions_HRR}
	p(n) =
	\frac{2\pi}{(24n-1)^{3/4}}
    \sum_{k \geq 1} \tfrac1k \,
    A_k(n) \,
	I_{3/2}\left(\frac{\pi\sqrt{24n-1}}{6k}\right),
\end{equation}
where $I_\alpha$ denotes the modified Bessel function of the first kind;
we call this a \emph{HRR expansion}.
The numbers $A_k(n)$, which we call \emph{HRR coefficients}, are given by
\begin{equation*}
	A_k(n) = 
	\sum_{\substack{0 \leq h<k \\(h, k)=1}}
	\exp\left(-2\pi i\left(\tfrac{hn}k + \tfrac12 s(h,k)\right)\right),
\end{equation*}
where $s(h,k)$ is a Dedekind sum.

The series in \eqref{eqn:partitions_HRR} can be used for computing the integer
$p(n)$, due to the existence of explicit bounds for the error term.
It becomes efficient only after the results of Lehmer (\cite{Leh38}) which show that,
instead of adding $k$-th roots of unity, the numbers $A_k(n)$ can be obtained by a much simpler computation.
More precisely, he obtained multiplicativity formulas
\begin{equation}
	\label{eqn:multip_intro}
	A_{k_1 k_2}(n) = A_{k_1}(n_1) \, A_{k_2}(n_2)
\end{equation}
for odd, coprime $k_1,k_2$ (and similar formulas in the even case);
here $n_1,n_2$ are determined by $k_1,k_2$ and $n$.
These formulas reduce the problem to computing $A_q(n)$ when $q$
is a prime power. In this case, by showing that the numbers $A_q(n)$ are given
by twisted Kloosterman sums and using closed formulas for evaluating them, he obtained that
\begin{equation}
	\label{eqn:primpep_intro}
	A_q(n) = \varepsilon_q(n) \sqrt{q} \cos(\alpha_q(n))
\end{equation}
for odd $q$ (and a similar formula for even $q$).
Here $\varepsilon_q(n) \in \{0,\pm 1,\pm 2\}$ and $\alpha_q(n) \in \mathbb R$ can be
obtained by simple calculations.

More recently, by efficiently implementing these formulas, Johansson (\cite{Joh12}) was able to
give an algorithm for computing $p(n)$ with complexity $O(n^{1/2+o(1)})$; this
is close to optimal, since the number of bits of $p(n)$ is $O(n^{1/2})$.

\medskip

The phenomena described above are not exclusive to the partition function.
The problem of obtaining HRR expansions for other partition functions has
been broadly studied in particular situations
(\cite{niven1940,hua1942,iseki1961,iskander2020}), as well as for the Fourier
coefficients of general eta-quotients of negative weight by Sussman
(\cite{Sus17}), of nonnegative weight by Chern (\cite{chern2019}), and for
harmonic Maass forms of negative weight in \cite{bringmann2011}. 

In the particular case of the overpartition function $\overline{p}(n)$, which can be described as
an eta-quotient by
\begin{equation*}
	\sum_{n \geq 0} \overline p(n) q^n =
	\frac{\eta(q^2)}{\eta^2(q)},
\end{equation*}
Barquero-Sanchez, Sirolli, Villegas-Morales and coauthors showed in \cite{us2023} that the analogous $\widetilde{A}_k(n)$ appearing
in its HRR expansion satisfy properties similar to those described above for the
partition function, and we obtained an error bound for it.

\medskip

The two main results of this article, which we now state,
show that formulas analogous to \eqref{eqn:multip_intro} and
\eqref{eqn:primpep_intro}, and to their counterparts in the case of the
overpartition function, hold for the coefficients $A^\delta_k(n)$ in the HRR series
describing the $n$-th coefficient of a general eta-quotient $\eta^\delta$ (see \eqref{eqn:etaq} for the definition).
Here we simplify the notation for convenience of the reader; we refer to their statements in the body of the article for full details.
\begin{theorem*}[Theorem~\ref{thm:multiplicativity}]
	Let $k_1,k_2$ be positive integers with $\gcd(k_1,k_2) = 1$.
	Assume that there exists a positive integer $\ell$ satisfying \eqref{eqn:mult_ell}.
	Then for every $n$ there exist (effectively computable) $n_1,n_2$ such that
	\begin{equation}
		\label{eqn:multip_intro2}
		A^\delta_{k_1 k_2}(n) = A^\delta_{k_1}(n_1) \, A^\delta_{k_2}(n_2).
	\end{equation}
\end{theorem*}

\begin{theorem*}[Theorem~\ref{thm:kloosterman}]
	For every positive integer $k$ there exist (effectively computable) integers
	$a,b,c$ and a character $\chi$ such that, for every $n$, we have
	\begin{equation}
		\label{eqn:primep_intro2}
		A^\delta_k(n) = i^c \cdot S_\chi(a-n,b;k),
	\end{equation}
	where $S_\chi$ is a twisted Kloosterman sum.
\end{theorem*}
We remark that this result is stronger than what Lehmer proved, since
\eqref{eqn:primep_intro2} holds
for every $k$ and not just for prime powers.
In particular, this gives an alternative explanation for an (unconditional)
analogue of \eqref{eqn:multip_intro2}, through the
multiplicativity of such sums.
Our proof, as in \cite{Leh38} and \cite{us2023}, is based on
well known congruences for
Dedekind sums and closed formulas for twisted Kloosterman sums.

\medskip

These results are complemented with Theorem~\ref{thm:explicit-error-bound}, which gives an error bound for
the HRR expansions of general eta-quotients of negative weight.
Such a bound is needed for using the HRR expansion to compute exactly the aforementioned Fourier coefficients.

\medskip

This article is organized as follows.
In the following section we describe the Hardy--Ramanu\-jan--Rademacher series
expansion for the coefficients of a given eta-quotient, together with the
notation that will be used throughout the article; we also give a list of
interesting examples.
In Section~\ref{sect:kloosterman_sums} we introduce twisted Kloosterman sums,
and give a series of properties that they satisfy, which will allow us to
compute them efficiently in (most of) the cases we are interested in.
In Sections~\ref{sect:hrr} and \ref{sect:multiplicativity} we prove our two main
results.
In Section~\ref{sect:error-bound} we prove
Theorem~\ref{thm:explicit-error-bound}, and describe explicitly its use in a
particular case.
Finally, in Section~\ref{sect:algorithm&examples} we synthesize our results
describing an algorithm for computing HRR coefficients, and we illustrate its
use in some examples.

\medskip

We conclude this introduction with a (random) sample of the efficiency of our
methods: when computing the number of 5-colored partitions of $10^6$, using
a plain implementation of Algorithm~\ref{alg:Ak} we obtained the needed HRR
coefficients in less than 9 seconds, whereas computing them by definition takes
more than 1 hour and 15 minutes.
See Examples \ref{ex:5-colored} and \ref{ex:r-colored} for details.

\subsection*{Acknowledgments}

This article has its roots in the workshop \emph{Teoría de Números en las Américas 2} workshop, held at CMO/BIRS in September 2024.
We are immensely thankful with the organizers for providing a perfect working atmosphere.

The first author thanks the Centro de Investigación en Matemática Pura y Aplicada (CIMPA) and the School of Mathematics of the University of Costa Rica for their administrative help and support during this project. This paper was written as part of research project C5139, led by the first author and registered with the Vicerrectoría de Investigación of the University of Costa Rica.

The second author acknowledges NSF grant DMS-2152182 for funding the travel to the above workshop.

The third author would like to thank the Guandong Technion Israel Institute of Technology for the support during his visit, which played a significant role in the completion of this article.

\section{Eta-quotients and Hardy--Ramanujan--Rademacher type series}
\label{sect:etaq}

Let $\eta$ denote Dedekind's eta function, given by
\begin{equation*}
	\eta(q) = q^{1/24} \prod_{n \geq 1} (1-q^n).
\end{equation*}
Let $\mathcal{M}$ be a finite subset of $\ZZ_{>0}$.
Fix a function $\delta : \mathcal{M} \to \ZZ$, and denote $\delta_m = \delta(m)$.
By $\eta^\delta$ we denote the eta-quotient
\begin{equation}
\label{eqn:etaq}
	\eta^\delta(q) = \prod_m \eta(q^m)^{\delta_m}.
\end{equation}
Here, and in what follows, $m$ runs through all elements of $\mathcal{M}$; unless otherwise stated.

We denote
\begin{equation}
	\label{eqn:n0}
	c_1 = -\tfrac12\sum_m \delta_m,
	\qquad
	n_0 = -\tfrac1{24} \sum_m m \cdot \delta_m.
\end{equation}
Then, since $q^{n_0} \eta^\delta(q)$ is holomorphic on 
$\set{q \in \mathbb{C}}{\abs{q} < 1}$,
it has a Fourier expansion
\begin{equation}
	\label{eqn:etaq-qexp}
	\eta^\delta(q) = q^{-n_0} \sum_{n \geq 0} a(n) q^n.
\end{equation}
Furthermore, when $n_0 \in \ZZ$ we get that $\eta^\delta(e^{2\pi i z})$ is a
weakly-holomorphic modular form of weight $-c_1$.

\begin{remark}
	\label{rmk:shift}
	Given $\delta$ as above and $m_0 \in \mathbb N$, defining $\delta_0(m m_0) =
	\delta_m(m)$ we get the shifted eta-quotient 
	\begin{equation*}
		\eta^{\delta_0}(q) = q^{-n_0 m_0} \sum_{n \geq 0} b(n) q^n,
	\end{equation*}
	where $b(n) = a(n/m_0)$ if $m_0 \mid n$, and $b(n) = 0$ otherwise.
	In particular, its Fourier coefficients have the same information as those
	of $\eta^\delta$.
\end{remark}

For $n \in \ZZ$ we consider the \emph{HRR coefficients} given, for each integer $k>0$, by
\begin{equation}
	\label{eqn:Akn}
	A^\delta_k(n) = 
	\sum_{\substack{0 \leq h<k \\(h, k)=1}}
	\exp\left(-2 \pi i\left(\tfrac{hn}k + 
	\tfrac12 \sum_m \delta_m \,
	s\left(\tfrac{mh}{\gcd(m,k)},\tfrac{k}{\gcd(m,k)}\right)\right)\right)
    .
\end{equation}
Here, given integers $a,b$ with $b>0$ we denote by $s(a,b)$ their Dedekind sum,
given by
\begin{align}
\label{eqn:dedekind_sum}
   s(a,b) = \sum_{j=0}^{b-1} \dparenth{\frac{j}{b}}
   \dparenth{\frac{aj}{b}},
\end{align}
where $\dparenth{x}$ is given by $x - \left\lfloor x \right\rfloor -1/2$ if $x \notin \ZZ$, and by $0$ otherwise.
We remark that for each $k$ the function $A^\delta_k$ is $k$-periodic.
Moreover, it is real valued; see Remark~\ref{rmk:dedekind}.

Since $s(a,1) = s(a,2) = 0 $ for every $a$ we get that
\begin{equation}
	\label{eqn:A1A2}
	A^\delta_1(n) = 1, \quad A^\delta_2(n) = (-1)^n.
\end{equation}

\medskip

The following result by Sussman (\cite{Sus17}) gives a Hardy--Ramanujan--Rademacher ex\-pan\-sion for the Fourier coefficients $a(n)$ given by \eqref{eqn:etaq-qexp}, in the case of negative weight.
Denote by $I_\alpha(x)$ the modified Bessel function of the first kind (see \eqref{eqn:bessel}).
For each integer $k > 0$ denote
\begin{align*}
	c_2(k) & = \prod_m \left(\frac{\gcd(m,k)}m\right)^{\delta_m/2},
	\qquad
	c_3(k) = -\sum_m \delta_m \frac{\gcd(m,k)^2}m, \\
	c_4(k) & = \min_m \left\{\frac{\gcd(m,k)^2}m\right\} - \frac{c_3(k)}{24}.
\end{align*}
When $n_0 \in \ZZ$, the number $-c_3(k)$ is, up to a positive known factor, the vanishing order of the modular form $\eta^\delta$ at the cusp $1/k$. 
\begin{theorem}[Sussman, 2017]
	\label{thm:sussman}
	If $c_1>0$ and $c_4(k) \geq 0$ for every $k$, then for $n>n_0$ we have
	\begin{multline}
		\label{eqn:sussman}
		a(n) = \frac{2\pi}{(24(n-n_0))^{(c_1+1)/2}} \\
		\cdot
		\sum_{c_3(k)>0} \tfrac1k \,
		c_2(k) \,
		c_3(k)^{(c_1+1)/2} \,
		A^\delta_k(n) \,
		I_{c_1+1}\left(\tfrac{\pi}k \sqrt{\tfrac23 c_3(k)(n-n_0)}\right).
	\end{multline}
\end{theorem}

\begin{remark}
    \label{rmk:c4}
	Let $M = \lcm\{m:m \in \mathcal{M}\}$. Then, since $\gcd{(m, k + Mt)} = \gcd{(m, k)}$ for every $t \in \Z$, we get that the functions $c_2, c_3$ and $c_4$ are
	$M$-periodic. In particular, using this, the hypothesis on the positivity of the values $c_4(k)$ for every $k$ appearing in the statement of Theorem \ref{thm:sussman} can be easily checked.
\end{remark}

\begin{remark}
    When $c_1 \leq 0$ Chern shows in \cite{chern2019} that, changing $I_{c_1+1}$ by $I_{-c_1-1}$, \eqref{eqn:sussman} can be used for estimating $a(n)$.
    More precisely, in this case the series is not convergent; the author gives error bounds for the remainder.
    \end{remark}

\begin{remark}
    Even though not given by eta-quotients, many of the examples of HRR expansions given in \cite{sills2010} involve the numbers $A^\delta_k(n)$ in \eqref{eqn:Akn}.
    \end{remark}

Table \ref{tab:examples} gives a list of examples of eta-quotients with Fourier coefficients
$a(n)$ giving well known partition functions.
In the table, the column labeled as $\delta$ lists the ordered pairs $(m, \delta_m)$ for $m \in \mathcal{M}$. 
In the first four the hypotheses of Theorem~\ref{thm:sussman} hold (the second one requires $r \leq 24$ for $c_4(k)$ to be non-negative for every $k$);
in the remaining ones, even though they satisfy that $c_1=0$, the identity given
by \eqref{eqn:sussman} remains valid.
This can be seen in the aforementioned references.

\begin{table}[ht]
\begin{tabular}{l|l|l}
	$\delta$ & Partition function & References
    \\\hline
	$\{(1,-1)\}$ & 			Partitions &  \cite{Leh38,Joh12} \\
	$\{(1,-r)\}$ &			$r$-colored partitions & \cite{iskander2020} \\
	$\{(1,-2),(2,1)\}$ &	Overpartitions & \cite{Zuc39,us2023} \\
	$\{(1,-1), (2,1), (4,-1)\}$ &	Partitions with no odd part repeated & \cite{sills2010}\\
	$\{(1,-1),(2,1)\}$ &	Partitions into odd parts & \cite{hua1942} \\
	$\{(1,-1),(pq,-1),$ &	Partitions into parts prime to $pq$, & \cite{iseki1961}\\
	$(p,1),(q,1)\}$     &	with $p,q$ distinct prime numbers &
\end{tabular}
\caption{Eta-quotients and partition functions}
\label{tab:examples}
\end{table}

All of our results, condensed mostly in the output of Algorithm~\ref{alg:Ak}, have been
thoroughly tested over these examples, with the aid of \texttt{SageMath} (\cite{sagemath}).

\section{Twisted Kloosterman sums}
\label{sect:kloosterman_sums}

Let $k$ be a positive integer.
Given nonzero integers $a,b$ modulo $k$ and a Dirichlet character $\chi$ modulo
$k$ we consider the twisted Kloosterman sum
\begin{equation*}
    S_\chi(a,b;k) = 
	\sum_{\substack{0\leq h<k \\(h, k)=1}} \chi(h) \exp{\left(
	\frac{2\pi i}{k}\left( ah +b\overline{h} \right) \right)}.
\end{equation*}
In such sums we always denote by $\overline{h}$ an inverse of $h$ modulo $k$.
When convenient, the sum can be taken over any set of representatives for $(\ZZ/k\ZZ)^\times$.

The following properties are well known for sums with trivial
character;
those proofs which can be easily extended to the twisted setting will be omitted.

\begin{proposition}
	With the above notation,
	\begin{enumerate}
		\item $S_\chi(a,b;k) = S_{\overline\chi}(b,a;k).$
            \item $\overline{S_\chi(a,b;k)} = S_{\overline\chi}(-a,-b;k).$
		\item If $k = k_1 k_2$ with $\gcd(k_1,k_2) = 1$, then
			\begin{equation}
				\label{eqn:k-mult}
					S_\chi(a,b;k) = 
					S_{\chi_1}\left(a/k_2,b/k_2;k_1\right)
					\cdot
					S_{\chi_2}\left(a/k_1,b/k_1;k_2\right),
			\end{equation}
			where $\chi = \chi_1 \cdot \chi_2$ under the natural isomorphism $(\ZZ/k\ZZ)^\times \simeq (\ZZ/k_1\ZZ)^\times \times (\ZZ/k_2\ZZ)^\times$, and the inverses are understood
			modulo $k_i$.
	\end{enumerate}
\end{proposition}

\begin{proposition}
	Given a nonzero integer $c$ with $\gcd(c,k) = 1$, we have that
	\begin{equation}
		\label{eqn:kloosterman_homog}
        {\chi}(c) \,
		S_\chi(ac,b;k) =
        S_\chi(a,bc;k).
	\end{equation}
\end{proposition}

For the remainder of this section we assume that $k = q = p^\alpha$, with $p$
prime, and $\beta \leq \alpha$ is such that $\chi$ is defined modulo $p^\beta$.
We denote by $\mathbf{1}$ the trivial character modulo $p$.

\medskip

The following is a simpler version of Selberg's identity for twisted 
Kloosterman sums.

\begin{proposition}
	Given integers $a,b$, write $\gcd(a,b,q) = p^\gamma$.
	If $\gamma < \alpha - \val_p(2)$, assume that $\beta \leq \alpha-\gamma$
	and that $\gamma = \val_p(b)$.
	Then
	\begin{equation}
		\label{eqn:selberg}
		S_\chi\left(a,b;p^\alpha\right)=
		\begin{cases}
			0,
			  &\gamma = \alpha, \, \chi \neq \mathbf{1},\\
			p^\alpha - p^{\alpha -1}
			  &\gamma = \alpha, \,\chi = \mathbf{1},\\
            p^\gamma \cdot \chi(b/p^\gamma) \cdot
			S_\chi\left(ab/p^{2\gamma},1;p^{\alpha-\gamma}\right)
			, &\gamma < \alpha - \val_p(2) .
		\end{cases}
	\end{equation}
	Finally, if $p = 2$ and $\gamma = \alpha-1$ then
	\begin{equation}
		\label{eqn:selberg2}
		S_\chi\left(a,b;p^\alpha\right) = 
		\begin{cases}
				0,& \chi \neq \mathbf{1},\\
				2^{\alpha-1} \, (-1)^{(a+b)/2^\gamma}, & \chi = \mathbf{1}.
		\end{cases}
	\end{equation}
        
\end{proposition}

\begin{proof}
	If $\gamma = \alpha$ then
	$ah+b\overline{h}=0\in \mathbb{Z}/p^\alpha\mathbb{Z}$ for all
	$h\in\left(\mathbb{Z}/p^\alpha\mathbb{Z}\right)^\times$;
	therefore\[S_\chi\left(a,b;p^\alpha\right)=\sum_{h\in\left(\mathbb{Z}/p^\alpha\mathbb{Z}\right)^\times}\chi(h)
    \exp\left(2\pi i(ah+b\overline{h})/p^\alpha\right)
    =\sum_{h\in\left(\mathbb{Z}/p^\alpha\mathbb{Z}\right)^\times}\chi\left(h\right),\]which
	proves \eqref{eqn:selberg} in this case.
	On the other hand, if $\gamma<\alpha$ then, writing $a=p^\gamma m$ and
	$b=p^\gamma n$ with $m,n\in\mathbb{Z}$ we get
	\begin{equation}
		\label{eqn:selberg_aux}
		S_\chi\left(a,b;p^\alpha\right)=
		\sum_{h\in\left(\mathbb{Z}/p^\alpha\mathbb{Z}\right)^\times}\chi\left(h\right)
		\exp\left(2\pi i(mh+n\overline{h})/p^{\alpha-\gamma}\right)
		.
	\end{equation}
	Each exponential
    depends only on the class
	of $mh+n\overline{h}$ modulo $p^{\alpha-\gamma}$, and $\chi(h)$ depends only on
	the class of $h$ modulo $p^{\alpha-\gamma}$, by the hypothesis on $\beta$.
	Therefore, since reduction modulo
	$p^{\alpha-\gamma}$ induces an epimorphism
	$\left(\mathbb{Z}/p^\alpha\mathbb{Z}\right)^\times\to
	\left(\mathbb{Z}/p^{\alpha-\gamma}\mathbb{Z}\right)^\times$ of index $p^\gamma$,
	we get that
	\[
		S_\chi(a,b;p^\alpha) = p^\gamma S_\chi\left(m,n;p^{\alpha-\gamma}\right).
	\]
	Then \eqref{eqn:selberg} follows from \eqref{eqn:kloosterman_homog}.

	Finally, when $p=2$ and $\gamma = \alpha - 1$, \eqref{eqn:selberg2} follows
	from \eqref{eqn:selberg_aux},
	since for odd $h$ we have
	\[
		\exp\left(2\pi i(mh+n\overline{h})/2\right)
		=
		\exp\left(\pi i(m+n)\right)
		=
		(-1)^{m+n}.
		\qedhere
	\]
\end{proof}

The above results reduce the computation of any twisted Kloosterman sum to the
computation of sums of the form $S_\chi(a,1;q)$.
We start with a well known case.

\begin{proposition}
    Let $a$ be an integer.
    Assume $p \mid a$.
    Then $S_\mathbf{1}\left(a,1;p\right) = -1$.
\end{proposition}
    
\begin{remark}
\label{rmk:satotate}
There does not exist a simple formula
for computing $S_\mathbf{1}(a,1;p)$ when $p\nmid a$.
\end{remark}

In what follows, for odd $p$ we denote
\begin{equation*}
    \epsilon_q = \begin{cases}
    1, & q\equiv 1 \pmod{4},\\
    i, & q\equiv 3 \pmod{4},
    \end{cases}
\end{equation*}
and we denote by $\krop$ the Legendre symbol modulo $p$.

\begin{proposition}
    \label{prop:acuadradop}
	Assume that $p > 2$. Let $a$ be an integer.
	\begin{enumerate}
		\item Assume $p \mid a$. Then
			$S_\krop\left(a,1;p\right)=\epsilon_p\sqrt{p}$.
		\item Assume $p \nmid a$. Then
			$S_\krop(a,1;p) = 0$, unless there exists $c$ such that
			$a \equiv c^2 \pmod{p}$.
	\end{enumerate}
 \end{proposition}

\begin{proof}
See \cite[Lem.* 2]{Leh38}.
\end{proof}

\begin{proposition}
    \label{prop:acuadrado}
	Assume that $\frac12\alpha \leq \beta <\alpha$, and that $\alpha > 2 + \beta$ if $p = 2$.
	Then $S_\chi(a,1;q) = 0$, unless $p \nmid a$ and there exists $c$ such that
	$a \equiv c^2 \pmod{q}$.
\end{proposition}

\begin{remark}
	When $\chi$ is not primitive modulo $\alpha$, the result holds letting
	$\beta = \alpha - 1$ when $p > 2$ and $\alpha > 1$, and $\beta = \alpha - 3$
	when $p = 2$ and $\alpha > 5$.
\end{remark}

\begin{proof}
	We consider the restricted case; the unrestricted case follows similarly.

    Using the identity \cite[(7)]{estermann1961} with
    $f(r) = \chi(r) \exp(2\pi i (ar + \overline{r})/p^\alpha)$
    (which requires the first hypothesis),
    and denoting $\gamma = \alpha - \beta$,
    we get that
    \begin{equation}
		\label{eqn:estermann}
        S_\chi(a,1;q) =
        \sum_{\substack{s=0 \\p\nmid s}}^{p^\beta-1}
        \chi(s)
        \exp\left(2\pi i(as + \overline{s})/p^\alpha\right)
        \sum_{t=0}^{p^\gamma-1}
        \exp\left(2\pi i(a - \overline{s}^2)t/p^\gamma\right).
    \end{equation}
    Then the result follows from the fact that the inner sum equals $p^\gamma$ or $0$,
    according as $a-\overline{s}^2$ is or not divisible by $p^\gamma$;
    by Hensel's lemma, this is equivalent to the existence of $c$ as in the statement.
\end{proof}

When $p = 2$ the following situations are not covered by the above proposition,
so we treat them separately.

\begin{proposition}
	\label{prop:salie2a}
	Assume that $p = 2$.
	Denote $\omega = \exp(2\pi i (a+1)/q)$.
	\begin{enumerate}
		\item If $\alpha = 1$ then
				$
				S_\chi(a,1;q) = \omega.
				$
		\item If $\alpha = 2$ then
				$
				S_\chi(a,1;q) = \left(1+(-1)^{a+1}\chi(3)\right) \cdot \omega.
				$
		\item If $\alpha = 3$ then
				$
				S_\chi(a,1;q) = 
				\left(1 + i^{a+1}\chi(3) + (-1)^{a+1}\chi(5) + (-i)^{a+1}\chi(7)\right)
				\cdot \omega
				$
		\item If $\alpha = 4$ and $\beta = 3$ then
				$S_\chi(a,1;q) = 0$ if $a \not \equiv 1 \pmod 2$; otherwise
			\[
				S_\chi(a,1;q) =
				\begin{cases}
					2\left(1-(-1)^m\chi(3)-\chi(5)+(-1)^m\chi(7)\right) \cdot \omega,
					   & a = 4m-1,\\
					2\left(1-i(-1)^m\chi(3)+\chi(5)-i(-1)^m\chi(7)\right) \cdot \omega,
					   & a = 4m+1.
				\end{cases}
			\]
	  	\item If $\alpha = 5$ and $\beta = 3$ then
				$S_\chi(a,1;q) = 0$ if $a \not \equiv 1 \pmod 4$; otherwise
			\[
				S_\chi(a,1;q) =
				\begin{cases}
					4\left(1+\sqrt{i}(-1)^m\chi(3)+\chi(5)-i\sqrt{i}^3(-1)^m\chi(7)\right) \cdot \omega,
					   & a = 8m-3,\\
					   4\left(1+i\sqrt{i}(-1)^m\chi(3)-\chi(5)-\sqrt{i}^3(-1)^m\chi(7)\right) \cdot \omega,
					   & a = 8m+1.
				\end{cases}
			\]
	\end{enumerate}
\end{proposition}

\begin{proof}
	We consider the case when $\alpha = 5$; the case when
	$\alpha = 4$ follows similarly, and the other cases are
	simpler.

	Since $\tfrac12 \leq \beta < \alpha$ the formula in \eqref{eqn:estermann}
	holds (with $\gamma = 2$).
	Since $\overline{s}^2 \equiv 1 \tmod 4$ when $2\nmid s$,
	as in the above proof, the inner sum in that formula (and hence
	$S_\chi(a,1;q)$) equals $0$ if $a \not \equiv 1 \tmod 4$; otherwise, 
	\begin{align*}
		S_\chi(a,1;q) & = 4
			\left(
				\chi(1) \omega - \chi(3) \omega^3 - \chi(5) \omega^5  + \chi(7) \omega^7
				   \right),
	\end{align*}
	from which the claim follows straightforwardly.
\end{proof}

We resume from Proposition~\ref{prop:acuadrado}.
When $a \equiv c^2 \tmod{q}$ with $p \nmid c$,
using \eqref{eqn:kloosterman_homog} we get that
\begin{equation*}
	S_\chi(a,1;q) =
	\overline{\chi}(u) \, S_\chi(u,u;q), 
\end{equation*}
where $u = c$ in the unrestricted case, and $u = (-1)^{\tfrac{c-1}2} c$ otherwise.
Using this, the following formulas will allow us to compute $S_\chi(a,1;q)$ in
the cases that we are interested in and that were not covered above.

\medskip

The following two results are due to Salié (see \cite{williams1971}).

\begin{proposition}
	\label{prop:s1-explicit}
	Assume that $p>2$ and $\alpha \geq 2$.
	Let $u$ be an integer with $p \nmid u$.
	Then
    \begin{equation*}
        S_\mathbf{1}(u,u;q) =
        \legendre{u}{q} 2\sqrt{q} \, \op{Re}\left(\epsilon_q \, e^{4\pi iu/q}\right).
    \end{equation*}
\end{proposition}

\begin{proposition}
	Assume that $p>2$.
	Let $u$ be an integer with $p \nmid u$.
    Then
    \[S_\krop\left(u,u;q\right)=
            \delta_q
            \epsilon_q
            \legendre{u}{q} 2\sqrt{q} \, \op{Re}\left(\delta_q\, e^{4\pi iu/q}\right)
            ,
            \]
    where we denote
    \begin{equation*}
        \delta_q =
        \begin{cases}
        1, & pq\equiv 1 \pmod{4},\\
        i, & pq\equiv 3 \pmod{4}.
        \end{cases}
    \end{equation*}
\end{proposition}

\begin{proposition}
	\label{prop:salie2b}
	Assume that $p=2$ and $\beta \geq 3$.
	Let $u$ be an integer with $2 \nmid u$.
	Denote $\omega = \exp(4\pi i u/q)$.
	\begin{enumerate}
        \item Assume that $\alpha = 2\beta$. Then
	\begin{equation*}
		S_\chi\left(u,u;q\right) =
            2^\beta\left(
			\omega\left(\chi(1)+i^u\chi\left(1+2^{\beta-1}\right)\right)
			+
			\overline{\omega}\left(\chi(-1)-i^u\chi\left(-1+2^{\beta-1}\right)\right)
			\right)
			.
	\end{equation*}
        \item Assume that $\alpha = 2\beta+1$. Then
	\begin{equation*}
		S_\chi\left(u,u;q\right) =
		2^\beta\left(
			2\omega i^u \sqrt{i}^{ut}\chi\left(1+2^{\beta -1}\right)
			+
		\overline{\omega} \chi\left(-1+2^{\beta-1}\right)
		\left(\sqrt{i}^{ut} + i^{-u} \sqrt{i}^{us}\right)
		\right)
		.
    \end{equation*}
    where $s = 5$ if $\beta = 3$ and $s = 1$ otherwise, 
    and $t = 3$ if $\beta = 3$ and $t = -1$ otherwise.
	\end{enumerate}
\end{proposition}

\begin{proof}
    The result can be obtained as in the proof given by \cite{williams1971} for
	Proposition~\ref{prop:s1-explicit} above:
    letting
    $\delta= \left\lfloor \alpha/2 \right\rfloor$
    and
    $\gamma = \alpha - \delta$,
    following the author we get that
    \begin{equation}
		\label{eqn:williams-nueva}
        S_\chi(u,u;q) =
        2^\delta
        \cdot
        \sum_v
        \chi(v)
        \exp(2\pi i u(v+\overline{v})/q),
    \end{equation}
    where $v$ ranges over the solutions to $v^2 \equiv 1 \tmod{2^\delta}$ 
    such that $0 < v < 2^\gamma$.

	When $\alpha = 2 \beta$ we have $\gamma = \delta = \beta$, and the
	solutions are $v = 1 + w 2^{\beta-1}$, with $w \in \{0,1\}$.
	Their inverses modulo $2^\alpha$ are 
	$\overline v = 1 - w 2^{\beta-1} + w 2^{\alpha-2}$.
	In particular, $v + \overline v = 2 + w 2^{\alpha-2}.$

	When $\alpha = 2 \beta + 1$ we have $\gamma = \beta+1$ and $\delta = \beta$.
	In this case the solutions and the underlying data are described in the
	following table, in which $w \in \{0,1\}$, and we denote $s_0 = -t$, and
	$s_1 = -s$:

	\begin{table}[ht]
	\begin{tabular}{|l|l|l|}
	\hline
	$v$ & $\overline v$ & $v + \overline v$ \\
	\hline
	\hline
	$1+w\cdot2^{\beta}$ & $1-w\cdot 2^{\beta}+w^2\cdot 2^{\alpha-1}$ & $2+w\cdot 2^{\alpha-1}$ \\
	\hline
	$1+2^{\beta-1}+w\cdot 2^{\beta}$ & $1-2^{\beta-1}+2^{\alpha-2}+(-w+2^{\beta-2}t)\cdot2^{\beta}$ & $2+2^{\alpha-2}+t\cdot2^{\alpha-3}$ \\
	\hline
	$-1+(1+w)\cdot 2^{\beta}$ & $-1-2^{\beta}+(1+w)^2\cdot 2^{\alpha-1}$ & $-2+(1+w)^2\cdot 2^{\alpha-1}$ \\
	\hline
	$-1+2^{\beta-1}+w\cdot2^{\beta}$ & $-1-2^{\beta-1}-w\cdot2^{\alpha-2}-(w+2^{\beta-2}s_w)\cdot2^{\beta}$ & $-2-s_w\cdot2^{\alpha-3}$ \\
	\hline
	\end{tabular}
	\end{table}

	Plugging these formulas for $v$ and $v + \overline v$ into
	\eqref{eqn:williams-nueva}, in the cases of even and odd $\alpha$
	respectively, gives the result straightforwardly.
\end{proof}

\begin{remark}
The explicit formulas given in Propositions~\ref{prop:salie2a} and
\ref{prop:salie2b}, together with Theorem~\ref{thm:kloosterman} below, imply
the formula given by Lehmer in \cite[Thm. 7]{Leh38} for $A_{2^\alpha}(n)$ in the
case of the partition function (which the author proves without the framework of
twisted Kloosterman sums).
\end{remark}

\section{The HRR coefficients as twisted Kloosterman sums}
\label{sect:hrr}

The HRR coefficients $A^\delta_k(n)$ given by \eqref{eqn:Akn} involve Dedekind sums $s(h,k)$ (see \eqref{eqn:dedekind_sum}).
We start this section by recalling identities, congruences and reciprocity
relations for Dedekind sums that will be needed in our proofs.

\begin{proposition}
\label{prop:dedekind1}
	Let $a,b$ be integers with $b>0$.
	\begin{enumerate}
		\item For every positive integer $q$,
			we have $s(qa, qb) = s(a, b)$.\label{item:s-cancelation}
        \label{item:s-homog}
		\item If $a \equiv a' \pmod{b}$,
			then $s(a, b) = s(a', b)$. \label{item:s-hhp}
            \item $s(-a,b) = -s(a,b)$.
	\end{enumerate}
\end{proposition}

\begin{proof}
    See, respectively, \cite[Thm. 1]{RW41}, \cite[(2.1)]{RW41} and \cite[(33.a)]{RG72}.
\end{proof}

\begin{remark}
\label{rmk:dedekind}
The first item shows why it suffices to consider Dedekind sums with relatively prime arguments.
The second item implies that in \eqref{eqn:Akn} the sum can be taken
over any set of representatives for $(\ZZ/k\ZZ)^\times$;
using this and the third item, by replacing $h$ by $-h$ in that sum we see that $A^\delta_k(n) \in \mathbb{R}$.
\end{remark}

\begin{proposition}\label{DedekindSumProperties}
Let $h, k$ be positive integers with $\gcd{(h, k)} = 1$.
\begin{enumerate}
\item The denominator of $s(h, k)$ is a divisor of $2k \gcd{(3, k)}$.
\label{item:s-entero}
\item If $\theta := \gcd{(3, k)}$, then
$
12hk\,s(h, k) \equiv h^2 + 1 \pmod{\theta k}.
$
\label{item:s-modk}
\item $12k\,s(h, k) \equiv 0 \pmod{3}$ if and only if $3 \nmid k$.
\label{item:s-mod3}
\item If $k$ is odd, then $12k\,s(h, k) \equiv k + 1 - 2 \legendre{h}{k} \pmod{8}.$
\label{item:s-mod8}
\item If $h$ is odd, then
	$12hk\,s(h, k) \equiv h^2 + k^2 + 3k + 1 + 2k\legendre kh \pmod{8k}$.
\label{item:s-mod8+}
\end{enumerate}
\end{proposition}

\begin{proof}
    See, respectively,
    \cite[Thm. 2]{RG72}, 
    \cite[Thm. 17]{RW41},
    \cite[(83)]{RG72}, 
    \cite[Thm. 18]{RW41}, and
    \cite[Thm. 19]{RW41}.
\end{proof}

\begin{proposition}\label{Rademacher-Whiteman}
Let $a, b, c$ be pairwise coprime positive integers.
Then
\begin{align*}
\left( s(ab, c) - \frac{ab}{12c} \right) + \left( s(bc, a) - \frac{bc}{12a} \right) - \left( s(b, ac) - \frac{b}{12ac} \right)+\frac{abc}{12} \in 2\mathbb{Z}.
\end{align*}
\end{proposition}

\begin{proof}
    See \cite[Thm. B]{RG72}.
\end{proof}

Fix an eta-quotient $\eta^\delta$, with the notation as in Section~\ref{sect:etaq}. 
We are going to assume, without loss of generality, that $24 \mid m$ for every
$m \in \mathcal{M}$ (which in turn implies that $n_0$, as defined in
\eqref{eqn:n0}, is an integer); see
Remark~\ref{rmk:shift}.

Let $k$ be a positive integer, and let $\lambda = \val_2(k)$.
For each $m$ denote $k_m = k/\gcd(m,k)$ and $m_k= m/\gcd(m,k)$.
Let
\begin{equation*}
	a_1 = -\sum_m \delta_m u_m m, \qquad
	b_1  = -\sum_m \delta_m u_m v_m \gcd(m,k),
\end{equation*}
where for each $m$ we let $u_m = 1$ if $3 \mid k_m$, and $u_m$ is an integer
such that $u_m \equiv 1 \mod(k_m), \, u_m \equiv 0 \mod(3)$ otherwise;
furthermore, $v_m$ denotes the inverse of $m_k$ modulo $\gcd(3,k_m) k_m$.
We also let
\begin{equation*}
	a_2 = -\sum_{2 \mid k_m} \delta_m m, \qquad
	b_2  = -\sum_{2 \mid k_m} \delta_m w_m \gcd(m,k)
		\left(k_m^2 + 3k_m + 1\right),
\end{equation*}
where for each $m$ such that $2 \mid k_m$ (so that $m_k$ is odd) we
denote by $w_m$ the inverse of $m_k$ modulo $2^{3+\lambda}$.

We remark that $3\mid a_1, \,3\mid b_1$ and $8 \mid a_2, \, 8\mid b_2$: this follows from the fact that, given a prime $p$, then $p\mid k_m$ if and only if $\val_p(k) > \val_p(m)$.

We let $t_1,t_2$ be integers such that
\begin{equation*}
	\begin{cases}
		t_1 \equiv 1 \pmod{3k/2^\lambda},\\
		t_1 \equiv 0 \pmod{2^{3+\lambda}},
	\end{cases}
	\qquad
	\begin{cases}
		t_2 \equiv 1 \pmod{2^{3+\lambda}},\\
		t_2 \equiv 0 \pmod{3k/2^\lambda},
	\end{cases}
\end{equation*}
Then $t_1 a_1 + t_2 a_2$ and $t_1 b_1 + t_2 b_2$ are integers modulo $24k$, and the above remark shows that they are divisible by 24.
Hence we let
\begin{equation*}
    a = (t_1 a_1 + t_2 a_2)/24, \quad
    b = (t_1 b_1 + t_2 b_2)/24,
\end{equation*}
which are integers modulo $k$.

Finally, denoting $\lambda_m = \val_2(k_m)$ and $k'_m = k_m /
2^{\lambda_m}$ for each $m$, we consider the function
\begin{equation*}
    \psi : 2\ZZ + 1 \to \{\pm1\},
    \qquad
	\psi(h) =
     	\prod_{\substack{2\nmid \delta_m,\, 2\mid k_m,\\ k'_m \equiv 1 \tmod4}}
        (-1)^{
        \tfrac{h m_k-1}2
        }
        \cdot
     	\prod_{\substack{2\nmid \delta_m,\, 2\mid k_m,\\ 2 \nmid \lambda_m}}
        (-1)^{
        \tfrac{(h m_k)^2-1}8
        }
        .
\end{equation*}
Then there exist $\alpha,\beta,\gamma \in \ZZ$ 
with $8 \mid \alpha+4\beta-\gamma$ such that
\begin{equation*}
	\psi(h) = (-1)^{\tfrac18(\alpha h^2 + 4\beta h - \gamma)}.
\end{equation*}
This implies that $\rho(h) = \psi(1) \psi(h)$ is a character modulo $8$.
Note that when $16\nmid k$, we have that $2\nmid k_m$ for every $m$,
hence $\psi = \rho \equiv 1$.
\medskip

The main result of this section shows that the HRR coefficients 
defined by \eqref{eqn:Akn} can be described by twisted Kloosterman sums.
We remark that, in the case of the partition function considered by \cite{Leh38}, a similar result is obtained only when $k$ is a prime power.
\begin{theorem}
\label{thm:kloosterman}
	Let $k$ be a positive integer.
	With the notation as above, let
	\begin{equation*}
    c = \tfrac12\sum_{2\nmid k_m} \delta_m (k_m-3) + 
			\sum_m
			\delta_m \legendre{(-1)^{k_m} \, m_k}{k'_m},
            \qquad
		\chi = 
			\prod_{2\nmid \delta_m} \legendre{\cdot}{k'_m}
			.
	 \end{equation*}
	 Then for every $n$
		$$A^\delta_k(n) = i^{c} \psi(1) \, S_{\chi\rho}(a-n,b;k).$$
\end{theorem}

\begin{proof}
We start by writing \eqref{eqn:Akn} as
\begin{equation}
 \label{eqn:g-Ak}
    A^\delta_k(n)=\sum_{\substack{0\leq h<k \\(h, k)=1}}
	\exp \left( \frac{\pi i \, g(h)}{12k} \right),
 \end{equation}
where for each $h$ as above we denote
$$g(h)=-12k\sum_{m} \delta_{m}s(mh,k)-24nh.$$
Proposition~\ref{DedekindSumProperties} (\ref{item:s-entero}) shows that $g(h) \in \ZZ$
(moreover, it is even).
We need to describe $g(h) \tmod{24k}$.
For this purpose, for each $m \in \mathcal{M}$ we denote $h_m = h m_k$,
and we note that, by Proposition~\ref{prop:dedekind1} (\ref{item:s-cancelation}), we have
\begin{equation}
	\label{eqn:shmkm}
	s(mh,k) = s(h_m,k_m).
\end{equation}

\begin{description}
    \item[Modulo $2^{3+\lambda}$]
Let $m$ be such that $2\nmid k_m$.
It is easy to see that $$\beta + 1 -2 \legendre{\alpha}{\beta} \equiv -6\beta
\legendre{-\alpha}{\beta} - 3\beta (\beta -3)  \pmod{8}$$ for all odd,
relatively prime $\alpha,\beta$.
Since $k_m$ is odd we can use
Proposition~\ref{DedekindSumProperties} (\ref{item:s-mod8})
on $s(h_m,k_m)$, which combined with \eqref{eqn:shmkm} gives
\begin{align*}
    12k_m \, s(mh, k) \equiv 
    -6k_m\legendre{-h_m}{k_m} - 3k_m(k_m - 3) \pmod{8}.
\end{align*}
Multiplying this equation by $k/k_m$ we get that
\begin{equation}
	\label{eqn:mod23laimpar}
    12k \, s(mh, k) \equiv 
	-6k\legendre{-h_m}{k_m} - 3k(k_m - 3) \pmod{2^{3+\lambda}}.
\end{equation}

Assume now that $m$ is such that $2\mid k_m$.
In this case we have that $h_m$ is odd.
Multiplying the congruence for $s(h_m,k_m)$ given by Proposition \ref{DedekindSumProperties}
(\ref{item:s-mod8+}) by $\gcd(m,k) \, \overline{h_m}$ and using \eqref{eqn:shmkm}
we have that
\begin{equation}
	\label{eqn:mod23lapar}
12k\,s(mh, k)
\equiv 
m h + \overline{h_m}\,
\biggl(
k k_m+
3k+\gcd(m,k) +2k
\legendre{k_m}{h_m}
\biggr)
\pmod{2^{3+\lambda}}.
\end{equation}

By quadratic reciprocity we have
\begin{equation*}
	\legendre{k_m}{h_m} =
	\legendre{2^{\lambda_m}}{h_m}
	\legendre{k'_m}{h_m} =
		(-1)^{\lambda_m\tfrac{h_m^2-1}8
			+ \tfrac{k'_m-1}2\tfrac{h_m-1}2}
			\legendre {h_m}{k'_m}.
\end{equation*}
Since $\overline{h_m} \equiv (-1)^{\tfrac{h_m-1}2} \tmod 4$,
we have that
$2k\overline{h_m} \equiv -6k \, (-1)^{\tfrac{h_m-1}2} \tmod{2^{3+\lambda}}$.
Then,
\begin{equation}
	\label{eqn:quadratic_recip}
	2k \overline{h_m}\legendre{k_m}{h_m} 
	\equiv
		-6k \,\psi_m(h) \legendre {h_m}{k'_m}
	\pmod{2^{3+\lambda}},
\end{equation}
with $\psi_m(h) = (-1)^{\lambda_m\tfrac{h_m^2-1}8
			+ \tfrac{k'_m+1}2\tfrac{h_m-1}2}$.

Combining \eqref{eqn:mod23laimpar}, \eqref{eqn:mod23lapar} and
\eqref{eqn:quadratic_recip} we get, adding up over $m$, that
\begin{multline}
    \label{eqn:gmod8}
	g(h) \equiv
	(a_2-24n) h + b_2 \overline{h} +\\
	6k\biggl(
		\,
	\sum_{2\nmid k_m} \tfrac12\delta_m (k_m-3)
	+ 
	\delta_m
	\legendre{-m_k}{k'_m}
	\legendre{h}{k'_m}
	+ 
	\\
	\sum_{2\mid k_m}
	\delta_m
	\legendre{m_k}{k'_m}
	\legendre{h}{k'_m}
	\psi_m(h)
	\biggr)
	\pmod{2^{3+\lambda}}
	.
\end{multline}

\item[Modulo $3k$]
By Proposition \ref{DedekindSumProperties} (\ref{item:s-modk})
we have that
\begin{equation}
	\label{eqn:mod3km}
	12k_m\, s(h_m,k_m) 
	\equiv u_m h_m + u_m \widehat{h_m}\pmod{3 k_m},
\end{equation}
where $\widehat{x}$ denotes the inverse of $x$ modulo $\gcd(3,k_m) k_m$.
With this notation,
\begin{equation*}
	\gcd(m,k) \widehat{h_m} \equiv \gcd(m,k) \widehat{m_k} \overline{h}
	\pmod{3k}.
\end{equation*}
Using this and \eqref{eqn:shmkm}, and multiplying \eqref{eqn:mod3km} by $k/k_m$
we get that
\begin{equation*}
	12k\, s(m h,k) \equiv
	u_m m h +
	u_m v_m \gcd(m,k) \overline{h}
	\pmod{3k},
\end{equation*}
and adding up over $m$ we conclude that
\begin{equation}
	\label{eqn:gmod3k}
	g(h) \equiv (a_1 - 24n) h + b_1 \overline{h} \pmod{3k}.
\end{equation}

\end{description}

Combining \eqref{eqn:gmod8} and \eqref{eqn:gmod3k} we get that
\begin{multline*}
	g(h) \equiv
	24(a-n) h + 24 b\overline{h} +\\
	6k\biggl(
		\,
	\sum_{2\nmid k_m} \tfrac12\delta_m (k_m-3)
	+ 
	\delta_m
	\legendre{-m_k}{k'_m}
	\legendre{h}{k'_m}
	+ 
	\sum_{2\mid k_m}
	\delta_m
	\legendre{m_k}{k'_m}
	\legendre{h}{k'_m}
	\psi_m(h)
	\biggr)
	\pmod{24k}
	.
\end{multline*}
Using this and \eqref{eqn:g-Ak}, and that
$\exp(\pm\tfrac{\pi i}2 t) = (\pm i)^t$ for every integer $t$,
we get that
\[
    A^\delta_k(n) = i^{c}\sum_{\substack{0\leq h<k \\(h, k)=1}} \chi(h)\psi(h) \exp{\left(
	\frac{2\pi i}{k}\left( (a-n)h +b\overline{h} \right) \right)}.
\]
This completes the proof.
\end{proof}

We conclude this section by remarking that the characters $\chi$ and $\rho$
appearing in Theorem~\ref{thm:kloosterman} satisfy that 
\begin{equation}
    \label{eqn:char_locales}
	\chi = \prod_{p \mid k,\,p\neq 2} \chi_p, \qquad
	\rho = \legendre{d_2}{\cdot},
\end{equation}
where $\chi_p \in \{\tri,\krop\}$ and $d_2 \in \{1,-1,2,-2\}$; these local
components can be easily be obtained from $k$ and $\delta$.
This will we useful when $k = q = p^\alpha$ with $p$ prime
(see Theorem \ref{thm:multiplicativity} below),
or when combining Theorem~\ref{thm:kloosterman} with \eqref{eqn:k-mult}.

\section{Multiplicativity of the HRR coefficients}
\label{sect:multiplicativity}

We now state one of our main results which shows that, under certain hypotheses, the numbers $A^\delta_k(n)$
satisfy multiplicativity properties, besides those inherited by
Theorem~\ref{thm:kloosterman} and \eqref{eqn:k-mult}.

Given relatively prime integers $k_1,k_2$, denote
\begin{equation*}
	u(k_1,k_2) = 
	\sum_m m\delta_m\left(
	\frac{k_1^2}{\gcd\left(m,k_1\right)^2}+
	\frac{k_2^2}{\gcd\left(m,k_2\right)^2}-
	\frac{(k_1 k_2)^2}{\gcd\left(m,k_1 k_2\right)^2}-1\right).
\end{equation*}

\begin{theorem}\label{thm:multiplicativity}Let $k_1,k_2 > 1$ be relatively prime integers, and let $k= k_1 k_2$.
Assume that there exists a positive integer $\ell$ relatively prime to $k$ such
that for every $m\in\mathcal{M}$ the following
hold:
\begin{equation}
    \label{eqn:mult_ell}
\begin{cases}\ell\equiv\gcd\left(m,k_1\right)^2
\quad\big(\mymod{\frac{k_2}{\gcd\left(m,k_2\right)}}\big),\\
\ell\equiv\gcd\left(m,k_2\right)^2
\quad\big(\mymod{\frac{k_1}{\gcd\left(m,k_1\right)}}\big).\end{cases}
\end{equation}
Let $\theta_1,\theta_2$ be such that $\theta_1 \theta_2 = 24$ and $\gcd(\theta_1 k_1, \theta_2 k_2) = 1$.
Then there exist $n_1,n_2\in\mathbb{Z}_{\geq 1}$ which are solutions
to
\begin{equation}
    \label{eqn:mult_n1n2}
\begin{cases}\ell(24n-u(k_1,k_2))\equiv 24n_1k_2^2\pmod{\theta_1k_1},\\
 \ell(24n-u(k_1,k_2))\equiv 24n_2k_1^2\pmod{\theta_2k_2}.\end{cases}
\end{equation}
 
Moreover, they satisfy $A^\delta_k(n)=A^\delta_{k_1}(n_1)A^\delta_{k_2}(n_2)$.

\end{theorem}

\begin{remark}
    When $\gcd(k,m) = 1$ for every $m \in \mathcal{M}$, the hypothesis of the theorem is satisfied with $\ell = 1$.
\end{remark}

\begin{remark}
    The hypothesis of this theorem holds in the cases of partitions and overpartitions (see Examples~\ref{ex:partitions} and~\ref{ex:overpartitions}).
    But there are situations where they do not and, moreover, there do not exist $n_1,n_2$ such that
    $A^\delta_k(n)=A^\delta_{k_1}(n_1)A^\delta_{k_2}(n_2)$ as in the conclusion: see Example~\ref{ex:no_odd_parts_repeated}.
    In such situations, we resort to \eqref{eqn:k-mult} (see the remarks following Algorithm~\ref{alg:Ak}).
    
    For a detailed analysis on to what extent the hypothesis holds in general, we refer to
    \cite[Sect. 3.2.3.3]{deRasis2026}.
\end{remark}

\begin{proof}

To show that there exist $n_1,n_2\in\mathbb{Z}_{\geq 1}$ solving the linear
system of congruences in the statement it suffices to see that both
$\gcd\left(24k_2^2,\theta_1k_1\right)$ and
$\gcd\left(24k_1^2,\theta_2k_2\right)$ divide
$24n-u(k_1,k_2)$.
This follows immediately from the fact they both divide $24$, and the fact that if $s,t\in\mathbb{Z}$ are relatively prime then $24 \mid s^2+t^2-s^2t^2-1$.

For simplicity, for each $m$ we denote $m_1 = \gcd(m,k_1), m_2 = \gcd(m,k_2)$ and $m_0 = m_1 m_2$.
Since $\gcd\left(k_1,k_2\right)=1$, letting $i \neq j$ we have that multiplication by $k_i$ permutes
$\left(\mathbb{Z}/k_j\mathbb{Z}\right)^\times$.
Therefore (see Remark~\ref{rmk:dedekind}, which we use repeatedly) we can write
\begin{equation*}
	A^\delta_{k_i}\left(n_i\right)  = 
	\sum_{\substack{0\leq h_i<k_i \\(h_i, k_i)=1}} \exp
\left[\pi I\left(-\frac{2n_ik_jh_i}{k_i}-\sum_{m}
\delta_ms\left(\frac{mk_jh_i}{m_i},\frac{k_i}{m_i}\right)\right)\right], \quad i \neq j,
\end{equation*}
letting $I$ denote momentarily the imaginary unit, to avoid confusions. Then
by the Chinese Remainder Theorem we get that
\begin{multline*}
A^\delta_{k_1}(n_1) A^\delta_{k_2}(n_2) = \sum_{\substack{0\leq h<k \\(h, k)=1}}\exp \left[\pi I\left(-\frac{2n_1k_2h}{k_1}-\frac{2n_2k_1h}{k_2}\right.\right.\\ - \left.\left.\sum_{m} \delta_m\left(s\left(\frac{mk_2h}{m_1},\frac{k_1}{m_1}\right)+s\left(\frac{mk_1h}{m_2},\frac{k_2}{m_2}\right)\right)\right)\right].\end{multline*}
Then, replacing $h$ by $\ell h$ in the sum defining $A^\delta_k(n)$ we get that, if
\begin{multline}\label{sufficientcondition}\frac{2n_1k_2h}{k_1}+\frac{2n_2k_1h}{k_2}-\frac{2nh\ell}{k_1k_2}+\sum_{m} \delta_m\left(s\left(\frac{mk_2h}{m_1},\frac{k_1}{m_1}\right)\right.\\+\left.s\left(\frac{mk_1h}{m_2},\frac{k_2}{m_2}\right)-s\left(\frac{mh\ell}{m_0},\frac{k_1k_2}{m_0}\right)\right)\in 2\mathbb{Z}\end{multline}
for every $h$ prime to $k$, then $A^\delta_k\left(n\right)=A^\delta_{k_1}\left(n_1\right)A^\delta_{k_2}\left(n_2\right)$.

Fix $h$ as above, and let $m \in \mathcal{M}$.
It is easy to see that, since $\gcd(k_1,k_2) = 1$, then
\[
    \gcd\left(\frac{k_1}{m_1},\frac{k_2}{m_2}\right) = 
    \gcd\left(\frac{mh\ell}{m_0},\frac{k_1}{m_1}\right) =
    \gcd\left(\frac{mh\ell}{m_0},\frac{k_2}{m_2}\right)
    = 1.
    \]
We can therefore apply Proposition~\ref{Rademacher-Whiteman} with
$a\coloneqq \frac{k_2}{m_2}, \,
b\coloneqq \frac{mh\ell}{m_0}, \,
c\coloneqq \frac{k_1}{m_1}$
to get that
\begin{multline*}s\left(\frac{mk_2h\ell}{m_1m_2^2},\frac{k_1}{m_1}\right)-\frac{\frac{mk_2}{m_2^2}h\ell}{12k_1}+s\left(\frac{mk_1h\ell}{m_2m_1^2},\frac{k_2}{m_2}\right)-\frac{\frac{mk_1}{m_1^2}h\ell}{12k_2}\\-s\left(\frac{mh\ell}{m_0},\frac{k_2k_1}{m_0}\right)+\frac{mh\ell}{12k_1k_2}+
\frac{\frac{k_1k_2}{m_0^2}mh\ell}{12}
\in 2\mathbb{Z}.\end{multline*}
Recalling that $m_i^2\equiv \ell\pmod{\frac{k_j}{m_j}}$ for $i\neq j$,
Proposition \ref{prop:dedekind1} \eqref{item:s-hhp} allows us to rewrite the above as
\begin{multline*}s\left(\frac{mk_2h}{m_1},\frac{k_1}{m_1}\right)-\frac{\frac{mk_2}{m_2^2}h\ell}{12k_1}+s\left(\frac{mk_1h}{m_2},\frac{k_2}{m_2}\right)-\frac{\frac{mk_1}{m_1^2}h\ell}{12k_2}\\-s\left(\frac{mh\ell}{m_0},\frac{k_2k_1}{m_0}\right)+\frac{mh\ell}{12k_1k_2}+\frac{\frac{k_1k_2}{m_0^2}mh\ell}{12}\in 2\mathbb{Z}.\end{multline*}
This gives that
\begin{equation*}
    - \frac{h\ell \, u(k_1,k_2)}{12 k_1 k_2}
    +
    \sum_{m} \delta_m\left(s\left(\frac{mk_2h}{m_1},\frac{k_1}{m_1}\right)+\right.\left.s\left(\frac{mk_1h}{m_2},\frac{k_2}{m_2}\right)-s\left(\frac{mh\ell}{m_0},\frac{k_1k_2}{m_0}\right)\right)
    \in 2\mathbb{Z},
\end{equation*}
and therefore \eqref{sufficientcondition} holds if and only if
\begin{align*}\frac{h}{12k_1k_2}\left(24n_1k_2^2+24n_2k_1^2-\ell(24n-u(k_1,k_2))\right)\in 2\mathbb{Z}.\end{align*}
This condition can be made independent from $h$, by showing that\[24k_1k_2 = (\theta_1 k_1)(\theta_2 k_2)\mid
24n_1k_2^2+24n_2k_1^2-\ell(24n-u(k_1,k_2)).\]
Since $\theta_jk_j\mid 24k_j^2$ for each $j\in\left\{1,2\right\}$, the above becomes equivalent to
\begin{align*}
	\theta_1k_1 & \mid 24n_1k_2^2-\ell(24n-u(k_1,k_2))\\
	\theta_2k_2 & \mid 24n_2k_1^2-\ell(24n-u(k_1,k_2)),
\end{align*}
which are true by definition of $n_1,n_2$.\end{proof}

We conclude this section with the following generalization of \cite[Thm. 3]{Leh38}.

\begin{corollary}
    Let $k>1$ be an odd integer.
    Assume that $2 \mid m$ for every $m \in \cal M$.
    Then
    \begin{equation*}
        A^\delta_{2k}(n) = (-1)^n A^\delta_k(n).
    \end{equation*}
\end{corollary}

\begin{proof}
    The positive integer $\ell = 4 + k_2$ solves \eqref{eqn:mult_ell}, and it is easy to verify that $n_1 = n_2 = n$ solve \eqref{eqn:mult_n1n2}.
	Then the formula follows, since $A^\delta_2(n) = (-1)^n$ (see \eqref{eqn:A1A2}).
\end{proof}

\section{Truncating the infinite series: an explicit bound for the error}
\label{sect:error-bound}

Fix an eta-quotient $\eta^\delta$, with the notation as in Section~\ref{sect:etaq}.
In this section we will prove an explicit bound for the error term $R(n,N)$ obtained when truncating the infinite series \eqref{eqn:sussman}, which is given for $n > n_0$ and $N \geq 1$ by
\begin{multline}\label{eq:R-definition}
		R(n,N) := \frac{2\pi}{(24(n-n_0))^{(c_1+1)/2}} \\
            \cdot
            \sum_{\substack{k=N+1 \\ c_3(k)>0}}^{\infty}
            \frac{1}{k} \,
		c_2(k) \,
		c_3(k)^{(c_1+1)/2} \,
		A^\delta_k(n) \,
		I_{c_1+1}\left(\frac{\pi}{k} \sqrt{\frac{2}{3} c_3(k)(n-n_0)}\right).
\end{multline}

\begin{remark}
    Such an error bound is needed when using \eqref{eqn:sussman} to compute the integer $a(n)$.
    More precisely, taking $N$ such that $\abs{R(n,N)} < 1/4$ and then computing each nonzero summand of \eqref{eqn:sussman} such that $1 \leq k \leq N$ with enough precision so that the partial sum can be given with error at most $1/4$, will give the value of $a(n)$:
    this is achieved by rounding the partial sum to the nearest integer.
    See \cite[Sect. 3.1]{Joh12} (resp. \cite[Sect. 8]{us2023}) for details in the case of the partition function (resp. the overpartition function).
\end{remark}

Recall that we denote by $I_\alpha$ the modified Bessel function of the first kind, given by
\begin{equation}
\label{eqn:bessel}
I_{\alpha}(x) = \sum_{m=0}^{\infty} \frac{1}{m!\,\Gamma(m+\alpha+1)}\left(\frac{x}{2}\right)^{2 m+\alpha}.
\end{equation}
We will use the following inequalities regarding $I_\alpha$.
\begin{proposition}
Assume that $x>0$. Then
\begin{align}
\label{eq:Bessel-bound}
I_1(x) & <\frac{x}2 \cosh(x), \\
I_{\alpha}(x) & \leq \left( \frac{x}{2} \right)^{\alpha} \sum_{v = 0}^{\infty} \frac{1}{v!(v+1)!} \left( \frac{x}{2} \right)^{2v}, \quad \alpha \geq 1.
\label{eq:Trivial-bound}
\end{align}
\end{proposition}

\begin{proof}
The second inequality follows directly from \eqref{eqn:bessel}, using that $v + \alpha + 1 \geq v + 2$ implies that $\Gamma(v + \alpha + 1) \geq \Gamma(v + 2) = (v + 1)!$.

For the first inequality, note that since
\begin{align*}
I_1(x)=\frac{x}{2} \sum_{v=0}^{\infty} \frac{1}{v!(v+1)!} \frac{x^{2 v}}{2^{2 v}} \quad \text{and} \quad \frac{x}{2}\cosh{(x)} = \frac{x}{2} \sum_{v=0}^{\infty} \frac{1}{(2v)!} x^{2v},
\end{align*}
it suffices to compare the coefficients of the corresponding series. We claim that 
\begin{equation*}
\frac{1}{v!(v+1)!} \cdot \frac{1}{2^{2 v}}<\frac{1}{(2 v)!} 
\end{equation*}
for $v \geq 1$.
This follows from the fact that the sequence $(h(v))_{v \geq 0}$ defined by
\begin{align*}
h(v) := \frac{2^{2v} v! (v+1)!}{(2v)!}
\end{align*}
is strictly increasing with $h(0) = 1$. This completes the proof of inequality \eqref{eq:Bessel-bound}.
\end{proof}

As was observed in Remark~\ref{rmk:c4}, the functions $c_2$ and $c_3$ are $M$-periodic, with $M := \lcm\{m:m \in \mathcal{M}\}$. Hence, we
can define the constants
\begin{equation*}
	C_2= \max_{c_3(k) > 0} c_2(k), \qquad
	C_3= \max_{c_3(k) > 0} c_3(k).
\end{equation*}

\begin{theorem}\label{thm:explicit-error-bound}
Let $n > n_0$ be a fixed positive integer, and suppose that $c_1 > 0$. Define
\begin{align*}
b(n):=\pi \sqrt{\tfrac{2}{3} C_3\left(n-n_0\right)}.
\end{align*}
Then for every $N \geq 1$ we have that $|R(n, N)| \leq M(n, N)$,  where 
\begin{align}\label{eqn:error-bound-definition}
    M(n, N) := \frac{2 \pi C_2}{c_1} \left(\frac{\pi C_3}{12}\right)^{(1+c_1)} \frac{N+1+c_1}{(N+1)^{1 + c_1}}\cosh{\left(\frac{b(n)}{N+1}\right)}.
\end{align}
Moreover, for every fixed $n > n_0$, we have that $M(n, N)$ is strictly decreasing as a function of $N$, and satisfies the asymptotic formula
\begin{align}\label{eqn:error_bound_asymptotic}
    M(n, N) = \frac{2 \pi C_2}{c_1} \left(\frac{\pi C_3}{12}\right)^{(1+c_1)} \frac{1}{(N+1)^{c_1}} \left( 1 + o(1) \right),
    \quad \text { as } N \rightarrow \infty.
\end{align}
\end{theorem}

\begin{proof}
Let $k$ be such that $c_3(k) > 0$. 
We have that $\abs{c_2(k)} \leq C_2$ and $\abs{c_3(k)} \leq C_3$.
We also trivially see that $\abs{A^\delta_k(n)} \leq k$.
Using these inequalities in the definition \eqref{eq:R-definition} for $R(n, N)$ we obtain
\begin{align}\label{eq:error-preliminary-bound}
    \abs{R(n, N)} &\leq 2\pi C_2 \left(\frac{C_3}{24(n - n_0)}\right)^{(1 + c_1)/2} \sum_{\substack{k=N+1 \\ c_3(k)>0}}^{\infty}  I_{1+c_1}\left(\frac{\pi}{k} \sqrt{\frac{2}{3} c_3(k)\left(n-n_0\right)}\right).
\end{align}
Since $I_{1 + c_1}(x)$ is increasing for $x > 0$ we get
that
\begin{align*}
I_{1+c_1}\left(\frac{\pi}{k} \sqrt{\frac{2}{3} c_3(k)\left(n-n_0\right)}\right) \leq I_{1+c_1}\left(\frac{b(n)}{k}\right).
\end{align*}
Using this inequality along with \eqref{eq:Trivial-bound}, and interchanging the order of summation, we have
\begin{multline}\label{eq:inifnite-bessel-sum-bound}
\sum_{\substack{k=N+1 \\ c_3(k)>0}}^{\infty} I_{1+c_1}\left(\frac{\pi}{k} 
\sqrt{\frac{2}{3} c_3(k)\left(n-n_0\right)}\right) 
 \leq \sum_{k = N+1}^{\infty} I_{1 + c_1}\left( \frac{b(n)}{k} \right) \\
\leq \sum_{k = N+1}^{\infty} \left( \frac{b(n)}{2k} \right)^{1 + c_1} \sum_{v = 0}^{\infty} \frac{1}{v! (v+1)!} \left( \frac{b(n)}{2k} \right)^{2v} 
= \sum_{v = 0}^{\infty} \frac{1}{v!(v+1)!} \sum_{k = N+1}^{\infty} \left( \frac{b(n)}{2k} \right)^{2v + 1 + c_1} 
\\ = \sum_{v = 0}^{\infty} \frac{1}{v!(v+1)!} \left( \frac{b(n)}{2} \right)^{2v + 1 + c_1} \sum_{k = N+1}^{\infty} \frac{1}{k^{2v + 1 + c_1}}.
\end{multline}

Now, recall that if $f \colon [1, \infty) \to \R$ is a positive, decreasing function, then
$$
\sum_{k = T}^{\infty} f(k) \leq f(T) + \int \limits_{T}^{\infty} f(x) \, dx,
$$
for every $T \in \Z_{\geq 1}$.
Applying this to the function $f(x) := \dfrac{1}{x^{2v + 1 + c_1}}$ with $T = N + 1$, and using that $c_1>0$, we get
\begin{align}\label{eq:integral-test-bound}
    \sum_{k = N + 1}^{\infty} \frac{1}{k^{2v + 1 + c_1}} &\leq \frac{1}{(N + 1)^{2v + 1 + c_1}} + \int_{N+1}^{\infty} \frac{1}{x^{2v + 1 + c_1}} \, dx \\ \notag
    &= \frac{1}{(N + 1)^{2v + 1 + c_1}} + \frac{1}{2v + c_1} \frac{1}{(N + 1)^{2v + c_1}}.
\end{align}
Then, combining \eqref{eq:inifnite-bessel-sum-bound} and \eqref{eq:integral-test-bound}, and using that
${I_1(x) = \sum_{v=0}^{\infty} \frac{1}{v!(v+1)!} \left( \frac{x}{2} \right)^{2v+1}}$,
we obtain
\begin{align*}
&\quad \sum_{\substack{k=N+1 \\ c_3(k)>0}}^{\infty} I_{1+c_1}\left(\frac{\pi}{k} \sqrt{\frac{2}{3} \left(n-n_0\right)}\right) 
\leq \sum_{v = 0}^{\infty} \frac{1}{v!(v+1)!} \left( \frac{b(n)}{2} \right)^{2v + 1 + c} \sum_{k = N+1}^{\infty} \frac{1}{k^{2v + 1 + c_1}} \\
&\leq \sum_{v = 0}^{\infty} \frac{1}{v!(v+1)!} \left( \frac{b(n)}{2} \right)^{2v + 1 + c_1} \left( \frac{1}{(N + 1)^{2v + 1 + c_1}} + \frac{1}{2v + c_1} \frac{1}{(N + 1)^{2v + c_1}}  \right)\\
&= \left( \frac{b(n)}{2(N+1)} \right)^{c_1} \sum_{v = 0}^{\infty} \frac{1}{v!(v+1)!} \left( \frac{b(n)}{2} \right)^{2v + 1} \left( \frac{1}{(N + 1)^{2v + 1}} + \frac{1}{2v + c_1} \frac{1}{(N + 1)^{2v}}  \right)\\
&\leq \left( \frac{b(n)}{2(N+1)} \right)^{c_1} \sum_{v = 0}^{\infty} \frac{1}{v!(v+1)!} \left( \frac{b(n)}{2} \right)^{2v + 1} \left( \frac{1}{(N + 1)^{2v + 1}} + \frac{1}{c_1} \frac{1}{(N + 1)^{2v}}  \right)\\
&= \left( \frac{b(n)}{2(N+1)} \right)^{c_1} \left( I_1\left( \frac{b(n)}{N+1} \right) + \frac{N+1}{c_1} I_1\left( \frac{b(n)}{N+1} \right) \right) \\
&= \frac{N + 1 + c_1}{c_1} \left( \frac{b(n)}{2(N+1)} \right)^{c_1}  I_1\left( \frac{b(n)}{N+1} \right),
\end{align*}
which, together with \eqref{eq:Bessel-bound} and \eqref{eq:error-preliminary-bound}, implies that

\begin{align}\label{eq:R-prelim-bound}
|R(n, N)| \leq 2 \pi C_2\left(\frac{C_3}{24\left(n-n_0\right)}\right)^{\left(1+c_1\right) / 2} \frac{N + 1 + c_1}{c_1} \left( \frac{b(n)}{2(N+1)} \right)^{1 + c_1}  \cosh{\left( \frac{b(n)}{N+1} \right)}.
\end{align}

Now, to simplify the terms that do not depend on $N$ in the inequality
\eqref{eq:R-prelim-bound}, note that by using the definition $b(n) :=\pi
\sqrt{\frac{2}{3} C_3\left(n-n_0\right)}$ we have
\begin{multline*}
\left(\frac{C_3}{24\left(n-n_0\right)}\right)^{\left(1+c_1\right) / 2} \left(
\frac{b(n)}{2} \right)^{1 + c_1} = \left(\frac{C_3 b(n)^2}{96\left(n-n_0\right)}\right)^{\left(1+c_1\right) / 2}\\
= \left( \frac{C_3}{96(n - n_0)} \cdot \pi^2 \cdot \frac{2}{3} \cdot C_3 \cdot (n - n_0) \right)^{(1 + c_1)/2}
= \left( \frac{\pi C_3}{12} \right)^{1 + c_1}.
\end{multline*}
Thus, using the previous simplification in \eqref{eq:R-prelim-bound}, we finally obtain
\begin{align*}
|R(n, N)| \leq M(n,N).
\end{align*}

Now we will analyze the asymptotic decay of
$M(n,N)$
as $N \to \infty$. Observe first that for every fixed $n > n_0$, we have that $M(n, N)$ is strictly decreasing as a function of $N$. This can be easily seen by checking individually that the terms
$$
\frac{N+1+c_1}{(N+1)^{1+c_1}} \quad \text{and} \quad \cosh \left(\frac{b(n)}{N+1}\right)
$$
are strictly decreasing. For instance by differentiating them with respect to $N$ one immediately sees that their derivatives are negative for every $N \geq 1$.
Finally, since $\cosh{(x)} = 1 + o(1)$ as $x \to 0$, we have
$$
\cosh{\left(\frac{b(n)}{N+1}\right)} = 1 + o(1)
$$
as $N \to \infty$. Similarly, as $N \to \infty$ we have
$$
\frac{N+1+c_1}{(N+1)^{1+c_1}} = \frac{1}{(N+1)^{c_1}} \left(1 + o(1)  \right).
$$
Therefore, using the last two asymptotic formulas in the formula defining $M(n, N)$ gives the asymptotic formula \eqref{eqn:error_bound_asymptotic}. This finishes the proof of the theorem.
\end{proof}

\begin{remark}
If $c_1 = 0$, we would have to separate the case of $v = 0$ in \eqref{eq:integral-test-bound}, which would lead to a logarithm for that term in the integration.
\end{remark}

\begin{remark}
	The trivial bound $\abs{A^\delta_k(n)} \leq k$ used in the proof can be improved in
	certain cases.
	After Theorem~\ref{thm:kloosterman}, there exist
	integers $a,b$ and characters $\chi,\rho$ such that
	$\abs{A^\delta_k(n)} = \abs{S_{\chi\rho}(a-n,b;k)}$.
	When $\gcd(b,k) = 1$ this twisted Kloosterman sum can be bounded by $2^{\omega(k)}\sqrt{k}$,
	using \eqref{eqn:k-mult} and the explicit formulas from
	Section~\ref{sect:kloosterman_sums} for the prime power case.
	Examples show that when $\gcd(b,k) \neq 1$ this bound does not hold; we
	have not been able to describe the behaviour of $\abs{A^\delta_k(n)}$ in these
	cases.
\end{remark}

\begin{example}
To exemplify Theorem \ref{thm:explicit-error-bound} we consider the case of overpartitions. For them, as listed in Table \ref{tab:examples}, the eta quotient is determined by $\delta = \{(1,-2),(2,1)\}$.
In \cite[Thm.~4.1]{us2023}, Barquero-Sanchez, Sirolli, Villegas-Morales and coauthors proved that the corresponding error term $R(n, N)$ satisfies the bound 
\[
|R(n, N)| \leq M_0(n, N),
\]
where sharper constants could be obtained by exploiting the particular structure of the overpartition generating function. Specifically, the bound is given by
\begin{align*}
M_0(n, N)=\frac{1}{4 \pi}\left(\frac{N+1}{n}\right)^{3 / 2}\left(\frac{\pi \sqrt{n}}{N+1} \cosh \left(\frac{\pi \sqrt{n}}{N+1}\right)+(2 N+1) \sinh \left(\frac{\pi \sqrt{n}}{N+1}\right)-2 \pi \sqrt{n}\right),
\end{align*}
and moreover $M_0(n, N)$ satisfies the asymptotic formula
\begin{align}\label{eqn:overpartition-old-asymptotic}
M_0(n, N)=\frac{\pi^2}{12} \frac{1}{\sqrt{N+1}}(1+o(1)) \quad \text { as } N \rightarrow \infty.
\end{align}

In contrast, Theorem \ref{thm:explicit-error-bound} provides an explicit upper bound that applies uniformly to a broad class of eta quotients. In the present example, we specialize the general parameters of Theorem \ref{thm:explicit-error-bound} to the overpartition case and compare the resulting bound with the one obtained in \cite{us2023}. This allows us to verify that the general bound is of comparable qualitative strength, while naturally involving a slightly weaker constant due to the need for uniform estimates in a more general setting.

Now, in this case the input for Theorem \ref{thm:explicit-error-bound} is given by
\begin{align*}
c_1 &= -\frac{1}{2} \sum_m \delta_m = -\frac{1}{2}\left( -2 + 1  \right) = \frac{1}{2},\\
n_0 &= -\frac{1}{24} \sum_m m \delta_m = -\frac{1}{24} \left( 1 \cdot(-2) + 2 \cdot (1) \right) = 0.
\end{align*}
Moreover, the period $M$ of the functions $c_2$ and $c_3$ is $M = \lcm\{ m \in \mathcal{M} \} = \lcm\{1, 2\} = 2$. Hence, we only need to compute $c_2(k)$ and $c_3(k)$ for $k = 1, 2$. Note that
\begin{align*}
c_2(1) &= 1/2, \quad c_2(2) = 1,\\
c_3(1) &= 3/2, \quad c_3(2) = 0.
\end{align*}
Therefore
\begin{align*}
	C_2 &= \max \{c_2(k) \colon \text{$1 \leq k \leq 2$ and $c_3(k) > 0$}
	\} = c_2(1) = 1/2, \\
	C_3 &= \max \{c_3(k) \colon \text{$1 \leq k \leq 2$ and $c_3(k) > 0$}
	\} = c_3(1) = 3/2,\\
b(n) &= \pi \sqrt{ \tfrac{2}{3}\,C_3(n - n_0)} = \pi \sqrt{n}.
\end{align*}

Plugging these values into the formulas \eqref{eqn:error-bound-definition} and \eqref{eqn:error_bound_asymptotic}, we find that Theorem \ref{thm:explicit-error-bound} implies that the error term satisfies the bound $|R(n, N)| \leq M(n, N)$ with
\begin{align*}
M(n, N) = \frac{\pi^{5/2}}{2^{7/2}} \frac{N+3/2}{(N+1)^{3/2}}
\cosh\!\left(\frac{\pi\sqrt{n}}{N+1}\right),
\end{align*}
and moreover, for every fixed $n \geq 1$, that $M(n, N)$ satisfies the asymptotic formula
\begin{align}\label{eqn:error_bound_asymptotic-overpartitions}
M(n, N) = \frac{\pi^{5/2}}{2^{7/2}} \frac{1}{\sqrt{N+1}} \left( 1 + o(1) \right)
\quad \text{as } N \to \infty.
\end{align}

In particular, comparing the asymptotic formulas \eqref{eqn:overpartition-old-asymptotic} and \eqref{eqn:error_bound_asymptotic-overpartitions}, we observe that the only difference lies in the leading constant: in \eqref{eqn:overpartition-old-asymptotic} the constant is $\frac{\pi^2}{12} \approx 0.8224670$, whereas in \eqref{eqn:error_bound_asymptotic-overpartitions} it is $\frac{\pi^{5/2}}{2^{7/2}} \approx 1.5462143$. This discrepancy is a natural consequence of the generality of Theorem \ref{thm:explicit-error-bound}, where some sharpness in the constants is necessarily sacrificed in order to obtain uniform bounds that apply to a wide class of eta quotients. Nevertheless, the comparison shows that the general bound retains essentially the same qualitative strength as the specialized result for overpartitions proved in \cite[Thm.~4.1]{us2023}.
\end{example}

\begin{example}
	\label{ex:5-colored}
We now consider the case of the (normalized) eta-quotient giving the $5$-colored
partitions: it is given by $\delta = \{(24,-5)\}$.
In this case the input for Theorem \ref{thm:explicit-error-bound} is given by
\begin{equation*}
c_1 = -\frac{1}{2} \sum_m \delta_m = 5/2, \qquad 
n_0 = -\frac{1}{24} \sum_m m \delta_m = 5.
\end{equation*}
Moreover, the period $M$ of the functions $c_2$ and $c_3$ is $M = \lcm\{ m \in
\mathcal{M} \} = 24$.
Therefore, a short calculation in \texttt{SageMath} gives
\begin{align*}
	C_2 &= \max \{c_2(k) \colon \text{$1 \leq k \leq 24$ and $c_3(k) >
	0$} \} = 32 \cdot 6^{5/2}, \\
	C_3 &= \max \{c_3(k) \colon \text{$1 \leq k \leq 24$ and $c_3(k) > 0$} \} = 120,\\
b(n) &= \pi \sqrt{ \frac{2}{3}C_3(n - n_0)} = 4\pi \sqrt{5(n-5)}.
\end{align*}

In order to use Sussman's formula to compute the number of $5$-colored
partitions of $n = 10^6$ we plug these parameters into
\eqref{eqn:error-bound-definition} and solve the inequality
\begin{equation*}
	M(24\cdot 10^6,N) = 2^{13} \cdot 15^{5/2}\cdot \pi^{9/2} \frac{N+7/2}{(N+1)^{7/2}} \cdot
	\cosh\left(\frac{4\pi\sqrt{5(n-5)}}{N+1}\right) < \frac12-\frac1{10}
\end{equation*}
(where we substract $1/10$ for safety to avoid possible floating-errors),
finding that the first solution is given by $N = 29,881$.
Thus the desired number is obtained by rounding the $N$-th partial sum of
$\eqref{eqn:sussman}$ to the nearest integer; it gives the 286-digit number
\begin{multline*}
1709349027900160426231759812453331777798866621250418728438426924737182606\\ 7755728323656323888351532514942462841895955264525681257154424656792274125\\
9847820143034212027874304489338548272548590372049411547222314153358818142\\
0284373293199200749115480647779688721463489973392452731512257715631
\end{multline*}
\end{example}

\section{Overview: Algorithm and examples}
\label{sect:algorithm&examples}

We summarize our main results using them in an algorithm for computing HRR coefficients, and describing its use in some examples.
We keep the notation from Section~\ref{sect:etaq}.

\SetKwComment{Comment}{/* }{ */}
\SetKw{KwReturn}{return}
\SetKwInput{KwData}{Input}
\SetKwInput{KwResult}{Output}

\begin{algorithm}[ht]
\caption{Evaluation of $A^\delta_{k}(n)$}
\label{alg:Ak}

\BlankLine

\KwData{An eta-quotient $\eta^\delta$, and integers $k\geq 1$, $n \geq 1$ }

\KwResult{$A^\delta_{k}(n)$, as defined in \eqref{eqn:Akn}}

\BlankLine

Factorize $k = p_{1}^{\alpha_{1}}\cdots p_{j}^{\alpha_{j}}$

$n_3 \gets n$

$k_2 \gets k$

$s \gets 1$

$i \gets 1$

\While{$s\neq 0$ and $k_2 \neq 1$}{
    $k_1 \gets p_{i}^{\alpha_{i}},\, q \gets k_1$
	\label{step:choose}
    
    $k_2 \gets k_2 / k_1$
    
	$n_1,n_2 \gets$ Theorem~\ref{thm:multiplicativity}, applied to $n_3$
	\label{step:multiplicativity}
    
    $s_1 \gets A^\delta_{q}(n_1)$, using Theorem~\ref{thm:kloosterman} and the
    formulas from Section~\ref{sect:kloosterman_sums} \label{step:primepow}

    $s \gets s \cdot s_1$

    $n_3 \gets n_2$
    
    $i \gets i + 1$
}
\KwReturn{$s$}
\end{algorithm}

\bigskip

Some remarks regarding this procedure:

\begin{enumerate}
	\item
		If Theorem~\ref{thm:multiplicativity} cannot be applied
		in Step~\ref{step:multiplicativity} (even choosing another prime power $k_1$ in Step~\ref{step:choose}),
		then we should use Theorem~\ref{thm:kloosterman} to compute $A^\delta_{k_2}(n_3)$, and
		then use multiplicativity formula~\eqref{eqn:k-mult}, in this and all of the remaining steps.
		See Example~\ref{ex:no_odd_parts_repeated}.
	\item 
		When looping this algorithm over $k$, the value of $s_1$
		should be stored and used when further needed.
	\item
		If the results from Section~\ref{sect:kloosterman_sums} do not give a
		closed formula for computing $s_1$
		(i.e., in the situation of Remark~\ref{rmk:satotate}: when $q = p > 3$ and $\chi = \tri$),
		it should be computed by definition.
		See Examples~\ref{ex:odd_parts}, \ref{ex:r-colored}.
	\item If the situation of Remark~\ref{rmk:satotate} does not arise then, due
		to the similarity of our procedure with that from \cite{Joh12}, we
		expect that a careful complexity analysis would yield a similar result as
		in op. cit., namely time $O(n^{1/2} \log^{4+o(1)} n)$.
 
	\item
		The algorithm will terminate, and return $0$, whenever the value $s_1$
		computed in Step~\ref{step:primepow} is $0$.
		According to Propositions~\ref{prop:acuadradop}
		and~\ref{prop:acuadrado}, this will be the situation roughly half of the
		times for each odd $q$.
		Thus, in general, we expect the $k$-th summand of \eqref{eqn:sussman} to be
		nonzero only $1/2^j$ of the times, where $j$ denotes the number of
		prime divisors of $k$.
\end{enumerate}

\medskip

The following are examples from Table~\ref{tab:examples}, normalized according
to Remark~\ref{rmk:shift}.
In all of them we used Remark~\ref{rmk:c4} to check that the hypothesis on
$c_4$ required by Theorem~\ref{thm:sussman} is verified, and to describe which
are the values of $k$ such that $c_3(k)>0$.

\begin{example} 
	\label{ex:partitions}
	Let $\delta = \{(24,-1)\}$, the normalization of $\{(1,-1)\}$.
    In this case we have that $c_3(k)>0$ for every $k$.
    The hypothesis of Theorem~\ref{thm:multiplicativity} is satisfied for every $k_1,k_2$, since $\# \mathcal{M} = 1$.
    The parameter $\ell$ given by this theorem equals $1$ when $\gcd(k_1 k_2, 6) = 1$, but this is not the situation in general. 
	The local characters described in \eqref{eqn:char_locales} are given by
	\[
		\chi_p =
		\begin{cases}
			\tri, & 2\mid \alpha - \val_p(3),\\
			\krop, & \text{otherwise},
		\end{cases}
		\qquad
		\psi =
		\begin{cases}
			\tri, & \quad \alpha < 4,\\
			\legendre{-2}{\cdot}, & \alpha \geq 4 \wedge 2\mid\alpha,\\
			-\legendre{-1}{\cdot}, & \alpha \geq 4 \wedge 2\nmid\alpha.
		\end{cases}
	\]

    This example (without the normalization) was treated with detail in \cite{Leh38}.
\end{example}

\begin{example}
	\label{ex:overpartitions}
	Let $\delta = \{(24,-2),(48,1)\}$, the normalization of $\{(1,-2),(2,1)\}$.
    In this case $c_3(k)>0$ if and only if $16 \nmid k$.
	The hypothesis of Theorem~\ref{thm:multiplicativity} is satisfied for every
	$k_1,k_2$ such that $\gcd(k_1,k_2)$ = 1 and $c_3(k_1 k_2) > 0$: assuming,
	without loss of generality, that $2 \nmid k_2$ (so that $16 \nmid k_1)$, the
	system in \eqref{eqn:mult_ell} is equal to
	\begin{equation*}
	\begin{cases}
		\ell \equiv \gcd\left(24,k_1\right)^2\quad
		\left(\mymod{\frac{k_2}{\gcd\left(3,k_2\right)}}\right),\\
		\ell \equiv \gcd\left(3,k_2\right)^2\quad
			\left(\mymod{\frac{k_1}{\gcd\left(24,k_1\right)}}\right),
	\end{cases}
	\end{equation*}
	which is compatible.
	The local characters described in \eqref{eqn:char_locales} are given by
	\[
		\chi =
		\begin{cases}
			\tri, & 2\mid \alpha - \val_p(3),\\
			\krop, & \text{otherwise},
		\end{cases}
		\qquad
		\psi =
		\begin{cases}
			\tri, & \quad \alpha < 5,\\
			\tlegendre{-2}{\cdot}, & \alpha \geq 5 \wedge 2\nmid\alpha,\\
			-\tlegendre{-1}{\cdot}, & \alpha \geq 5 \wedge 2\mid\alpha.
		\end{cases}
	\]

	This example (without the normalization) was treated with
	detail in \cite{us2023}.
\end{example}

\begin{example}
    \label{ex:no_odd_parts_repeated}
    Let $\delta = \{(24,-1),(48,-1),(96,1)\}$ be the eta-quotient which normalizes 
	$\{(1,-1),(2,1),(4,-1)\}$.
    Here $c_3(k)>0$ if and only if $ 32 \nmid k$.
	Given $k_1,k_2$ such that $\gcd(k_1,k_2) = 1$ and $c_3(k_1 k_2) > 0$,
	assuming, without loss of generality, that $2 \nmid k_2$ (so that $32 \nmid
	k_1)$, a simple calculation shows that there exists $\ell$ such that
	\eqref{eqn:mult_ell} is compatible if
	and only if
	\begin{equation*}
		\gcd(16,k_1) \equiv \gcd(8,k_1) \quad
			\left(\mymod{\tfrac{k_2}{\gcd(3,k_2)}}\right).
	\end{equation*}
	In particular, when $k_1 = 5$ and $k_2 = 16$
	we cannot use Theorem~\ref{thm:multiplicativity};
	moreover, we verified numerically that there do not exist $n_1,n_2$ such that
	$A^\delta_{80}(16) = A^\delta_5(n_1) \cdot A^\delta_{16}(n_2)$.
	Nevertheless, for every $n$, using
	Theorem~\ref{thm:kloosterman} and \eqref{eqn:k-mult} we can write
	\begin{align*}
		A^\delta_{80}(n)
		= &
		-i \cdot S_{\tlegendre{-10}{\cdot}}(49-n,43;80)
		\\
		= &
		-i \cdot S_{\tlegendre{5}{\cdot}}\left(1\cdot(49-n),1\cdot43;5\right) \cdot
		S_{\tlegendre{-2}{\cdot}}\left(13\cdot(49-n),13\cdot43;16\right),
	\end{align*}
	and then use the closed formulas from Section~\ref{sect:kloosterman_sums}
	for computing $S_\chi(a,b;q)$ when $q=5,16$.
	
\end{example}

The following two examples show that there are situations where the efficiency
of our method depends on the efficiency for computing Kloosterman sums
$S_\tri(a,1;p)$.

\begin{example}
	\label{ex:odd_parts}
    Let $\delta = \{(24,-1),(48,1)\}$, the normalization of
	$\{(1,-1),(2,1)\}$.
    In this case $c_3(k)>0$ if and only if $16 \nmid k$.
	Moreover, the system \eqref{eqn:mult_ell} is the same as in
	Example~\ref{ex:overpartitions}, hence it is compatible.
	The local characters $\chi_p$ described in \eqref{eqn:char_locales} are
	trivial, whereas
	\[
		\psi =
		\begin{cases}
            \tri, & \alpha < 4,\\
			\tlegendre{-2}{\cdot}, & \alpha = 4,\\
			-\tlegendre{2}{\cdot}, & \alpha \geq 5.
		\end{cases}
	\]

\end{example}

\begin{example}
	\label{ex:r-colored}
	Let $r \in \mathbb N$, and let $\delta = \{(24,-r)\}$, the normalization of
	$\{(1,-r)\}$.
	Here the situation is the same as in Example~\ref{ex:partitions}; with the
	exception that, when $r$ is even,
	the local characters $\chi_p$ and $\psi$ described in \eqref{eqn:char_locales} are trivial.
\end{example}

\newcommand{\etalchar}[1]{$^{#1}$}


\begin{thebibliography}{BSCVR{\etalchar{+}}23}

\bibitem[BO12]{bringmann2011}
Kathrin Bringmann and Ken Ono.
\newblock Coefficients of harmonic {M}aass forms.
\newblock In {\em Partitions, {$q$}-series, and modular forms}, volume~23 of {\em Dev. Math.}, pages 23--38. Springer, New York, 2012.

\bibitem[BSCVR{\etalchar{+}}23]{us2023}
Adrian Barquero-Sanchez, Gabriel Collado-Valverde, Nathan~C. Ryan, Eduardo Salas-Jimenez, Nicol\'as Sirolli, and Jean~Carlos Villegas-Morales.
\newblock Efficient computation of the overpartition function and applications.
\newblock {\em J. Math. Anal. Appl.}, 528(1):Paper No. 127472, 22, 2023.

\bibitem[Che19]{chern2019}
Shane Chern.
\newblock Asymptotics for the {F}ourier coefficients of eta-quotients.
\newblock {\em Journal of Number Theory}, 199:168--191, 2019.

\bibitem[DR26]{deRasis2026}
Juan~Pablo De~Rasis.
\newblock Efficient computation of fourier coefficients of eta-quotients \& first-order definability of campana points and darmon points.
\newblock \url{https://drive.google.com/file/d/1_K2G0-17yO42SHew4AEZ236Yny49f9Ae/view?usp=sharing}, 2026.

\bibitem[Est61]{estermann1961}
T.~Estermann.
\newblock On {K}loosterman's sum.
\newblock {\em Mathematika}, 8:83--86, 1961.

\bibitem[Hua42]{hua1942}
Loo-keng Hua.
\newblock On the number of partitions of a number into unequal parts.
\newblock {\em Trans. Amer. Math. Soc.}, 51:194--201, 1942.

\bibitem[IJT20]{iskander2020}
Jonas Iskander, Vanshika Jain, and Victoria Talvola.
\newblock Exact formulae for the fractional partition functions.
\newblock {\em Res. Number Theory}, 6(2):Paper No. 20, 17, 2020.

\bibitem[Ise61]{iseki1961}
Sh\^o Iseki.
\newblock Partitions in certain arithmetic progressions.
\newblock {\em Amer. J. Math.}, 83:243--264, 1961.

\bibitem[Joh12]{Joh12}
Fredrik Johansson.
\newblock Efficient implementation of the {H}ardy-{R}amanujan-{R}ademacher formula.
\newblock {\em LMS J. Comput. Math.}, 15:341--359, 2012.

\bibitem[Leh38]{Leh38}
D.~H. Lehmer.
\newblock On the series for the partition function.
\newblock {\em Trans. Amer. Math. Soc.}, 43(2):271--295, 1938.

\bibitem[Niv40]{niven1940}
Ivan Niven.
\newblock On a certain partition function.
\newblock {\em Amer. J. Math.}, 62:353--364, 1940.

\bibitem[Rad38]{R38}
Hans Rademacher.
\newblock On the partition function $p(n)$.
\newblock {\em Proceedings of the London Mathematical Society}, 2(1):241--254, 1938.

\bibitem[RG72]{RG72}
Hans Rademacher and Emil Grosswald.
\newblock {\em Dedekind sums}.
\newblock The Carus Mathematical Monographs, No. 16. Mathematical Association of America, Washington, D.C., 1972.

\bibitem[RW41]{RW41}
Hans Rademacher and Albert Whiteman.
\newblock Theorems on {D}edekind sums.
\newblock {\em Amer. J. Math.}, 63:377--407, 1941.

\bibitem[Sil10]{sills2010}
Andrew~V. Sills.
\newblock Rademacher-type formulas for restricted partition and overpartition functions.
\newblock {\em Ramanujan J.}, 23(1-3):253--264, 2010.

\bibitem[Sus17]{Sus17}
Ethan Sussman.
\newblock Rademacher series for $\eta$-quotients, 2017.

\bibitem[{The}25]{sagemath}
{The Sage Developers}.
\newblock {\em {S}ageMath, the {S}age {M}athematics {S}oftware {S}ystem ({V}ersion 10.8)}, 2025.

\bibitem[Wil71]{williams1971}
Kenneth~S. Williams.
\newblock Note on the {K}loosterman sum.
\newblock {\em Proc. Amer. Math. Soc.}, 30:61--62, 1971.

\bibitem[Zuc39]{Zuc39}
Herbert~S. Zuckerman.
\newblock On the coefficients of certain modular forms belonging to subgroups of the modular group.
\newblock {\em Trans. Amer. Math. Soc.}, 45(2):298--321, 1939.

\end{thebibliography}
\end{document}